\documentclass[11pt]{amsart}
\usepackage{amsmath,amsthm,amssymb, amscd, amsfonts}
\usepackage[all]{xy}
\usepackage{leftidx}
\usepackage[latin1]{inputenc}
\usepackage{enumerate}

\theoremstyle{plain}

\newtheorem{theorem}{Theorem}[section]
\newtheorem{corollary}[theorem]{Corollary}

\newtheorem{proposition}[theorem]{Proposition}
\newtheorem{lemma}[theorem]{Lemma}

\theoremstyle{definition}
\newtheorem{definition}[theorem]{Definition}

\newtheorem{example}[theorem]{Example}

\theoremstyle{remark}
\newtheorem{remark}[theorem]{Remark}

\numberwithin{equation}{section}\theoremstyle{plain}

\newcommand{\I}{\mathcal{I}}

\renewcommand{\1}{\textbf{1}}

\newcommand{\A}{{\mathcal A}}
\newcommand{\B}{{\mathcal B}}
\newcommand{\C}{{\mathcal C}}
\newcommand{\D}{{\mathcal D}}
\newcommand{\F}{{\mathcal F}}

\newcommand{\Z}{{\mathcal Z}}
\newcommand{\Zz}{{\mathbb Z}}

\newcommand{\Ii}{\mathfrak{I}}

\newcommand{\E}{{\mathcal E}}
\newcommand{\U}{{\mathcal U}}

\newcommand{\Rep}{\operatorname{Rep}}

\newcommand\Irr{\operatorname{Irr}}
\newcommand\FPdim{\operatorname{FPdim}}

\newcommand\vect{\operatorname{Vec}}
\newcommand\svect{\operatorname{sVec}}
\newcommand\id{\operatorname{id}}

\newcommand\Hom{\operatorname{Hom}}
\newcommand\rk{\operatorname{rk}}

\begin{document}
\title[generalized near-group fusion categories]{Structure, examples and classification for generalized near-group fusion categories}

\author{Jingcheng Dong}
\address{College of Mathematics and Statistics, Nanjing University of Information Science and Technology, Nanjing 210044, China}
\email{jcdong@nuist.edu.cn}

\author{Hua Sun}
\address{Department of Mathematics, Yangzhou University, Yangzhou, Jiangsu 225002, China}
\email{d160028@yzu.edu.cn}



\date{\today}

\begin{abstract}
We describe the structure of a generalized near-group fusion category and present an example of this class of fusion categories which arises from the extension of a Fibonacci category. We then classify slightly degenerate generalized near-group fusion categories. We also prove a structure result for braided generalized Tambara-Yamagami fusion categories. 
\end{abstract}

\maketitle

\section{Introduction}\label{sec1}
Let $\C$ be a fusion category, and let $G$ be the group generated by invertible simple objects of $\C$. Then there is an action of $G$ on the set of non-isomorphic non-invertible simple objects by left tensor product. If this action is transitive then $\C$ is called a generalized near-group fusion category in \cite{Thornton2012Generalized}.

\medbreak
Let $\C$ be a generalized near-group fusion category and let $X,Y$ be non-invertible simple objects in $\C$. Then $X\otimes X^*$ and $Y\otimes Y^*$ admit the same decompositions (see Section \ref{sec3}):
\begin{equation}
\begin{split}
X\otimes X^*=Y\otimes Y^*=\bigoplus_{h\in \Gamma}h\oplus k_1X_{1}\oplus \cdots\oplus k_nX_{n},
\end{split}\nonumber
\end{equation}
where $\{X_1,\cdots,X_n\}$ is a full list of non-isomorphic non-invertible simple objects of $\C$, $\Gamma$ is the stabilizer of $X$ under the action of $G$. In this case, we say that $\C$ is a generalized near-group fusion category of type $(G,\Gamma,k_1,\cdots,k_n)$.

\medbreak
In the thesis \cite{Thornton2012Generalized}, Thornton obtained some basic results and classified generalized near-group fusion categories when they are symmetric or modular. In this paper, we will continue to study generalized near-group fusion categories.

\medbreak
Let $\C$ be a generalized near-group fusion category of type $(G,\Gamma,k_1,\cdots,k_n)$.  Roughly speaking, we can divide generalized near-group fusion categories $\C$ into two classes according to whether $(k_1,\cdots,k_n)$ is $0$ or not. If $(k_1,\cdots,k_n)=0$ then $\C$ is a generalized Tambara-Yamagami fusion category. This class of fusion categories were introduced in \cite{liptrap2010generalized}, and then were further studied in \cite{natale2013faithful}. 

If $(k_1,\cdots,k_n)\neq0$ then the adjoint subcategory $\C_{ad}$ of $\C$ is not pointed, and hence it is also a generalized near-group fusion category, see Proposition \ref{subcategory}. Moreover, there is a 1-1 correspondence between the non-pointed fusion subcategories and the subgroups of the universal grading group, see Proposition \ref{categorytype}. Hence $\C_{ad}$ is the smallest non-pointed fusion subcategory of $\C$. Let $X\in \C_{ad}$ be a non-invertible simple object. Then $X\otimes X^*$ also admits the decomposition $\oplus_{h\in \Gamma}h\oplus k_1X_{1}\oplus \cdots\oplus k_nX_{n}$. So the adjoint subcategory $(\C_{ad})_{ad}$ of $\C_{ad}$ is also non-pointed, since $(k_1,\cdots,k_n)\neq0$. It follows that the universal grading group of $\C_{ad}$ is trivial, since $\C_{ad}$ is the smallest non-pointed fusion subcategory of $\C$. This shows that $\C$ is an extension of a smaller generalized near-group fusion category with trivial grading. Therefore, the problem of classifying generalized near-group fusion categories is reduced to classifying such fusion categories with trivial universal grading.

\medbreak
In general, an extension of  a generalized near-group fusion category with trivial grading is not necessary a generalized near-group fusion category, see Remark \ref{exten-GNG}. Hence it is interesting to decide whether an extension of a  generalized near-group fusion category is again a generalized near-group fusion category or not.  A Fibonacci category is a non-pointed fusion category of rank $2$. It can be viewed as the easiest generalized near-group fusion category with $(k_1,\cdots,k_n)\neq0$. Our result shows that any extension of a Fibonacci category is a generalized near-group fusion category, see Remark \ref{exten-Fib}. In addition, any extension $\C$ of a Fibonacci category $\F$  admits an exact factorization $\C=\F\bullet \C_{pt}$, where $\C_{pt}$ is the largest pointed fusion subcategory of $\C$. The notion of an exact factorization of a fusion category was introduced in \cite{gelaki2017exact}.

\medbreak
As we have mentioned, modular generalized near-group fusion categories have been classified by Natale and Thornton.  The next level of complexity is to classify slightly degenerate ones. Suppose that $\C$ is a slightly degenerate fusion category. Then $|G(\C)|=|\U(\C)|$ or $|G(\C)|=2|\U(\C)|$, where $G(\C)$ is the group generated by invertible simple objects of $\C$, $\U(\C)$ is the universal grading group of $\C$, see Proposition \ref{slig_degen}. Using this fact, we obtain that a slightly degenerate generalized near-group fusion category with $(k_1,\cdots,k_n)=0$ is in fact an extension of a pointed fusion category of rank $2$. Hence we can adopt the result from \cite{DNS2019}. If $\C$ is a slightly degenerate fusion category with $(k_1,\cdots,k_n)\neq0$, then the adjoint subcategory $\C_{ad}$ is either a Fibonacci category, or a slightly degenerate category of the form $\C(\mathfrak{psl}_2,q^t,8)$ with $q=e^{\frac{\pi i}{8}}$ and $(t,2)=1$, see \cite{bruillard2017classification} for the details on fusion categories $\C(\mathfrak{psl}_2,q^t,8)$. Then we can get the decomposition of $\C$ by the  M\"{u}ger decomposition theorem. Our main result is listed below (The fusion category $\Ii_{N, \zeta}$ is constructed in \cite{DNS2019} ).

\begin{theorem}
Let $\C$ be a slightly degenerate generalized near-group fusion category. Then $\C$ is exactly one of the following:

(1)\, $\C\cong \Ii_{N, \zeta} \boxtimes \B$, for some $N>2$, where  $\zeta \in k^\times$ is a primitive $2^N$th root of 1,  and $\B$ is a braided pointed fusion category.

(2)\, $\C\cong \C_{ad}\boxtimes \C_{pt}$,  where $\C_{ad}$ is a Fibonacci category.

(3)\, $\C\cong\C_{ad}\boxtimes \B$,  where $\C_{ad}$ is a slightly degenerate fusion category of the form $\C(\mathfrak{psl}_2,q^t,8)$ with $q=e^{\frac{\pi i}{8}}$ and $(t,8)=1$, $\B$ is a non-degenerate pointed fusion category.
\end{theorem}

\medbreak
In \cite{natale2013faithful}, non-degenerate generalized Tambara-Yamagami fusion categories have been classified by Natale. In this paper, we continue to study braided generalized Tambara-Yamagami fusion categories. A first observation shows that the classification of braided generalized Tambara-Yamagami fusion categories $\C$ is reduced to the case when $\FPdim(\C)$ is a power of $2$.  It turns out we can get second main result below.

\begin{theorem}\label{BGTY-001}
Let $\C$ be a braided generalized Tambara-Yamagami fusion category of dimension $2^n$. Then

(1)\, Suppose that $\C$ is integral and  $cd(\C)=\{1,2^i\}$. Then $\C$ is a $G$-equivariantization of a pointed fusion category, where $G$ is an Abelian group of order $2^{2i-1}$. In particular, $\C$ is group-theoretical.

(2)\, Suppose that $\C$ is not integral and  $cd(\C)=\{1,2^i\sqrt{2}\}$. Then $\C$ is a $G$-equivariantization of some $\Ii_{N, \zeta} \boxtimes \B$, where $G$ is an Abelian group of order $2^{2i-1}$,  and $\B$ is a braided pointed fusion category.
\end{theorem}

\medbreak
The paper is organized as follows. In Section \ref{sec2},  we discuss some basic notions and results on fusion categories that will be used throughout the paper.  In Section \ref{sec3}, we describe the structure of a generalized near-group fusion category. In Section \ref{sec4}, we present an example of a generalized near-group fusion category which arises from the extension of a Fibonacci category. In Section \ref{sec5}, we classify slightly degenerate generalized near-group fusion categories. In Section \ref{sec6}, we study the structure theorem for braided generalized Tambara-Yamagami fusion categories. Throughout this paper, we will work over an algebraically closed field $k$ of characteristic $0$.

\section{Preliminaries}\label{sec2}
A fusion category $\C$ is a $k$-linear semisimple rigid tensor category with finitely many isomorphism classes of simple objects, finite-dimensional vector space of morphisms and the unit object $\1$ is simple.

\subsection{Invertible simple objects}\label{sec2.1}
Let $\C$ be a fusion category and let $K(\C)$ be the Grothendieck ring of $\C$. Then the set $\Irr(\C)$ of isomorphism classes of simple objects in $\C$ is a basis of $K(\C)$. The Frobenius-Perron dimension of $X\in \Irr(\C)$ is the Frobenius-Perron eigenvalue of the matrix of left multiplication by $X$ in the basis $\Irr(\C)$.

The Frobenius-Perron dimension of $\C$ is defined by
$$\FPdim(\C)=\sum_{X\in\Irr(\C)}\FPdim(X)^2.$$

A simple object $X\in \C$ is called invertible if $X\otimes X^*\cong \1$, where $X^*$ is the dual of $X$. This implies that $X$ is invertible if and only if $\FPdim(X)=1$. A fusion category $\C$ is called pointed if every element in $\Irr(\C)$ is invertible. Let $\C_{pt}$ be the fusion subcategory generated by all invertible simple objects in $\C$. Then $\C_{pt}$ is the largest pointed fusion subcategory of $\C$.

\begin{lemma}\label{action_of_G(C)}
Let $\C$ be a fusion category, and let $g$ be an invertible simple object of $\C$. Then $g\otimes X$ is again a simple object for every simple object $X$ of $\C$.
\end{lemma}
\begin{proof}
If $g\otimes X$ admits a decomposition $U\oplus V$ then
$$X\cong g^*\otimes g\otimes X\cong g^*\otimes (U\oplus V)\cong(g^*\otimes U)\oplus (g^*\otimes V).$$
This contradicts the fact that $X$ is simple.
\end{proof}

Let $\C$ be a pointed fusion category and let $X,Y\in\Irr(\C)$. Then $X\otimes Y$ also lies in $\Irr(\C)$ by Lemma \ref{action_of_G(C)}.
This property endows $G:=\Irr(\C)$ the structure of a finite group with multiplication given by tensor product. The inverse of $X\in G$ is its dual $X^*$. The pointed fusion category $\C$ is classified by the group $G$ and a cohomology class $\omega \in H^3(G,k^*)$, see \cite{DiPaR1991}. We denote such a fusion category by $\vect_{G}^{\omega}$.

Let $\C$ be a fusion category, and let $G(\C)$ be the group generated by $\Irr(\C_{pt})$. Then there is an action of $G(\C)$ on the set $\Irr(\C)$ by left tensor product by Lemma \ref{action_of_G(C)}. Let $G[X]$ be the stabilizer of any $X\in \Irr(\C)$ under this action. Then
$$1=\dim\Hom(g\otimes X,X)=\dim\Hom(g,X\otimes X^*)$$
shows that $g$ appears in the decomposition of $X\otimes X^*$ with multiplicity $1$. Hence for any simple object $X$, we have a decomposition
\begin{equation}\label{decom1}
\begin{split}
X\otimes X^*=\bigoplus_{g\in G[X]}g\oplus\sum_{Y\in \Irr(\C)-G[X]}  \dim\Hom(Y,X\otimes X^*)Y.
\end{split}
\end{equation}

\medbreak
For any $g\in G(\C)$, $X\in\Irr(\C)$, $(g\otimes X)\otimes (g\otimes X)^*\cong g\otimes (X\otimes X^*)\otimes g^{-1}$. This fact implies the following lemma.
\begin{lemma}\label{conjugate}
For any $g\in G(\C)$, $X\in\Irr(\C)$, we have $G[g\otimes X]=gG[X]g^{-1}$.
\end{lemma}

\subsection{Group grading of a fusion category}\label{sec2.2}
Let $G$ be a finite group. A $G$-graded fusion category is a fusion category $\C$ admitting a direct sum of full abelian subcategories $\C=\oplus_{g\in G}\C_g$ such that $(\C_g)^*=\C_{g-1}$ and $\C_g\otimes\C_h\subseteq\C_{gh}$ for all $g,h\in G$. This grading is called faithful if $\C_g\neq 0$ for all $g\in G$. We say $\C$ is a $G$-extension of $\D$ if $\C=\oplus_{g\in G}\C_g$ admits a faithful grading such that $\D$ is equivalent the trivial component $\C_e$.

\medbreak
By \cite[Proposition 8.20]{etingof2005fusion}, if $\C$ is a $G$-extension of $\D$ then
\begin{equation}\label{FPdimgrading}
\begin{split}
\FPdim(\C_g)=\FPdim(\C_h),\,\, \FPdim(\C)=|G| \FPdim(\D), \forall g,h\in G.
\end{split}
\end{equation}

Let $\C$ be a fusion category. The adjoint subcategory $\C_{ad}$ of $\C$ is the full subcategory generated by simple objects in $X\otimes X^*$ for all $X\in \Irr(\C)$. It is well known that $\C$ has a canonical faithful grading $\C=\oplus_{g\in \mathcal{U}(\C)}\C_g$ with trivial component $\C_e=\C_{ad}$. This grading is called the universal grading of $\C$, and $\mathcal{U}(\C)$ is called the universal grading group of $\C$. Let $\C=\oplus_{g\in G}\C_g$ be any grading of $\C$. Then the trivial component $\C_e$ contains $\C_{ad}$ by \cite[Corollary 3.7]{gelaki2008nilpotent}.

\subsection{Braided fusion categories}\label{sec2.3}
A braided fusion category $\C$ is a fusion category admitting a braiding $c$, where the braiding is a family of natural isomorphisms: $c_{X,Y}$:$X\otimes Y\rightarrow Y\otimes X$ satisfying the hexagon axioms for all $X,Y\in\C$.

\medbreak
Let $\C$ be a braided fusion category and $\D\subseteq$ be a fusion subcategory of $\C$ . The centralizer $\D'$ of $\D$ is the full subcategory generated by $\{X\in \C| c_{Y,X}c_{X,Y}=\id_{X\otimes Y},\forall\ Y\in \D\}$. The M\"{u}ger center $\Z_2(\C)$ of $\C$ is the centralizer $\C'$ of $\C$. The fusion category $\C$ is called non-degenerate if $\Z_2(\C)$ is equivalent to the trivial category $\vect$. The fusion category $\C$ is called slightly degenerate if $\Z_2(\C)$ is equivalent to the category $\svect$ of super-vector spaces.

\medbreak
A braided fusion category $\C$ is called symmetric if $\Z_2(\C)=\C$. A symmetric fusion category $\C$ is called Tannakian if it is equivalent to the category $\Rep(G)$ of finite-dimensional representations of a finite group $G$, as braided fusion categories. The following result is due to Drinfeld et al.

\begin{theorem}{\cite[Corollary 2.50]{drinfeld2010braided}}\label{SymmCat}
Let $\C$ be a symmetric fusion category. Then one of following holds:

(1)\,  $\C$ is Tannakian;

(2)\,  $\C$ is an $\mathbb{Z}_2$-extension of a Tannakian subcategory.
\end{theorem}

Let $\B$ be a fusion subcategory of a braided fusion category $\C$. The commutator of $\B$ is the fusion subcategory $\B^{co}$ generated by all simple objects $X\in \C$ such that $X\otimes X^*\in \B$.

%The lemma below follows immediately from the definition of a commutator.
%
%\begin{lemma}\label{commutator}
%For any fusion subcategory $\K$ of a braided fusion category $\C$, the commutator $\K^{co}$ always contains the largest pointed fusion subcategory $\C_{pt}$.
%\end{lemma}

\begin{lemma}\label{centralizer_C_ad}
Let $\C$ be a braided fusion category. Then the largest pointed fusion subcategory $(\C_{ad})_{pt}$ of $\C_{ad}$ is a symmetric category.
\end{lemma}
\begin{proof}
Applying \cite[Proposition 3.25]{drinfeld2010braided}, we obtain that $(\C_{ad})'=(\C')^{co}$ which contains $\C_{pt}$ as a fusion subcategory, by the definition of a commutator. The M\"{u}ger center of $\C_{ad}$  is $\C_{ad}\bigcap (\C_{ad})'$ which thus contains $\C_{ad}\bigcap \C_{pt}=(\C_{ad})_{pt}$. Hence $(\C_{ad})_{pt}$ is a symmetric category.
\end{proof}

%\begin{proposition}\label{Adjo_Point}
%Let $\C$ be a braided fusion category. Then $\C_{ad}\subseteq (\C_{pt})'$.
%\end{proposition}
%\begin{proof}
%Suppose first that $\C$ is non-degenerate. Then $\C_{ad}= (\C_{pt})'$ by \cite[Corollary 3.27]{drinfeld2010braided}.
%
%Now we suppose that $\C$ is an arbitrary braided fusion category. The braiding of $\C$ induces a canonical embedding of braided fusion categories $\C\hookrightarrow \Z(\C)$. Hence, we may identify $\C$ with a fusion subcategory of $\Z(\C)$. We therefore have $\C_{pt}\subseteq \Z(\C)_{pt}$ and $\C_{ad}\subseteq \Z(\C)_{ad}$. This implies that $(\C_{pt})'\supseteq (\Z(\C)_{pt})'=\Z(\C)_{ad}\supseteq\C_{ad}$. The  equality holds true because $\Z(\C)$ is non-degenerate, also by \cite[Corollary 3.27]{drinfeld2010braided}. This completes the proofs.
%\end{proof}

%Let $\C$ and $\D$ be two fusion categories with $\Irr(\C)=\{X_1,X_2,\cdots,X_n \}$ and $\Irr(\D)=\{Y_1,Y_2,\cdots,Y_m \}$. The Deligne's tensor product $\C\boxtimes \D$ is a fusion category with simple objects $\Irr(\C\boxtimes \D)=\{X_i\boxtimes Y_j\mid 1\leq i\leq n,1\leq j\leq m\}$ and morphisms $\Hom(X_i\boxtimes Y_j,X_s\boxtimes Y_t)=\Hom(X_i, X_s)\otimes \Hom(Y_j, Y_t)$.
%
%Let $X_1\boxtimes Y_1$ and $X_2\boxtimes Y_2$ be two simple objects of $\C\boxtimes \D$. Then we have
%$$(X_1\boxtimes Y_1)\otimes (X_2\boxtimes Y_2)=(X_1\otimes X_2)\boxtimes (Y_1\otimes Y_2).$$

The following theorem is known as the \emph{M\"{u}ger  Decomposition Theorem}, since it is due to M\"{u}ger \cite[Theorem 4.2]{muger2003structure} when $\C$ is modular.

\begin{theorem}{\cite[Theorem 3.13]{drinfeld2010braided}}\label{MugerThm}
Let $\C$ be a braided fusion category and $\D$ be a non-degenerate subcategory of $\C$. Then $\C$ is braided equivalent to $\D\boxtimes \D'$, where $\D'$ is the centralizer of $\D$ in $\C$.
\end{theorem}

For a pair of  fusion subcategories $\A,\B$ of $\D$, we use $\A\vee \B$ to denote the smallest fusion subcategory of $\C$ containing $\A$ and $\B$. The following result will be frequently used in our proof.
\begin{lemma}\cite[Corollary 3.11]{drinfeld2010braided}\label{double_centralzer}
Let $\C$ be a braided fusion category. If $\D$ is any fusion subcategory of $\C$ then $\D''=\D\vee\mathcal{Z}_2(\C)$.
\end{lemma}

\begin{proposition}\label{non_degenerate}
Let $\C$ be a non-degenerate fusion category. Then $|G(\C)|=|\mathcal{U}(\C)|$.
\end{proposition}
\begin{proof}
By \cite[Theorem 3.14]{drinfeld2010braided} and the non-degeneracy of $\C$, we have

\begin{equation}
\begin{split}
&\FPdim(\C_{ad})\FPdim(\C_{ad}')\\
=&\FPdim(\C)\FPdim(\C_{ad}\cap \mathcal{Z}_2(\C))\\
=&\FPdim(\C).
\end{split}\nonumber
\end{equation}

By \cite[Corollary 3.27]{drinfeld2010braided}, we have $\FPdim(\C_{pt})=\FPdim(\C_{ad}')$. This induces an equation $\FPdim(\C_{ad})\FPdim(\C_{pt})=\FPdim(\C)$.
On the other hand, $\FPdim(\C)=|\mathcal{U}(\C)| \FPdim(\C_{ad})$ (see Subsection \ref{sec2.2}). Hence $|G(\C)|=\FPdim(\C_{pt})=|\mathcal{U}(\C)|$.
\end{proof}

\begin{proposition}\label{slig_degen}
Let $\C$ be a slightly degenerate braided fusion category. Then one of the following holds true.

(1)\, $|G(\C)|=|\mathcal{U}(\C)|$. If this is the case then $\mathcal{Z}_2(\C) \nsubseteq \C_{ad}$.

(2)\, $|G(\C)|=2|\mathcal{U}(\C)|$. If this is the case then $\mathcal{Z}_2(\C_{ad})=\mathcal{Z}_2(\C_{ad}^{'})$ contains the category $\svect$.
\end{proposition}
\begin{proof}
By \cite[Proposition 3.29]{drinfeld2010braided}, $\C_{ad}^{'}=\C_{pt}$. Hence \cite[Theorem 3.14]{drinfeld2010braided} shows that
\begin{equation}\label{eq1}
\begin{split}
&\quad\FPdim(\C_{ad})\FPdim(\C_{pt})\\
&=\FPdim(\C_{ad})\FPdim(\C_{ad}^{'})\\
&=\FPdim(\C)\FPdim(\C_{ad}\cap\mathcal{Z}_2(\C)).
\end{split}\nonumber
\end{equation}

Since $\C_{ad}\cap\mathcal{Z}_2(\C)$ is a fusion subcategory of $\mathcal{Z}_2(\C)= \svect$, we have $\C_{ad}\cap \mathcal{Z}_2(\C)=\vect$ or $\svect$.

\medbreak
If $\C_{ad}\cap \mathcal{Z}_2(\C)=\vect$  then $$\FPdim(\C_{ad})\FPdim(\C_{pt})=\FPdim(\C)=\FPdim(\C_{ad})|\mathcal{U}(\C)|.$$
Hence $|G(\C)|=\FPdim(\C_{pt})=|\mathcal{U}(\C)|$. In this case $\mathcal{Z}_2(\C) \nsubseteq \C_{ad}$ since $\C_{ad}\cap \mathcal{Z}_2(\C)=\vect$.

\medbreak
If $\C_{ad}\cap \mathcal{Z}_2(\C)=\svect$ then $$\FPdim(\C_{ad})\FPdim(\C_{pt})=2\FPdim(\C)=2\FPdim(\C_{ad})|\mathcal{U}(\C)|.$$
Hence $|G(\C)|=\FPdim(\C_{pt})=2|\mathcal{U}(\C)|$. Moreover, in this case we have
$$\mathcal{Z}_2(\C_{ad})=\C_{ad}\cap \C_{ad}^{'}=\C_{ad}\cap \C_{pt}\supseteq \C_{ad}\cap \svect=\svect.$$

By Lemma \ref{double_centralzer}, we have
\begin{equation}\label{eq2}
\begin{split}
\mathcal{Z}_2(\C_{ad}^{'})&=\C_{ad}^{'}\cap \C_{ad}^{''}=\C_{ad}^{'}\cap (\C_{ad}\vee \mathcal{Z}_2(\C))\\
&=(\C_{ad}^{'}\cap \C_{ad})\vee \mathcal{Z}_2(\C)= \mathcal{Z}_2(\C_{ad})\vee \mathcal{Z}_2(\C)\\
&=\mathcal{Z}_2(\C_{ad}).
\end{split}\nonumber
\end{equation}

The third equation follows from \cite[Lemma 5.6]{drinfeld2010braided}, $\C$ being braided  and the fact that $\mathcal{Z}_2(\C)=\svect$ is contained in $\C_{ad}^{'}=\C_{pt}$. The last equation follows from the fact that $\mathcal{Z}_2(\C)=\svect$ is a fusion subcategory of $\mathcal{Z}_2(\C_{ad})$.
\end{proof}

\begin{corollary}\label{slig_degen_triv_grad}
Let $\C$ be a slightly degenerate braided fusion category with trivial universal grading. Then we have

 (1)\, $\C_{pt}=\svect$.

 (2)\, The stabilizer $G[X]$ is trivial for every $X\in \Irr(\C)$.
\end{corollary}
\begin{proof}
 (1)\, By Proposition \ref{slig_degen} and the assumption that $\mathcal{U}(\C)$ is trivial, we have $\FPdim(\C_{pt})=1$ or $2$. On the other hand, $\mathcal{Z}_2(\C)=\svect$ is a fusion subcategory of $\C_{pt}$. Hence $\FPdim(\C_{pt})=2$ and $\C_{pt}=\svect$.

 (2)\, Let $\delta$ be the unique non-trivial simple object generating $\mathcal{Z}_2(\C)=\C_{pt}$. By \cite[Lemma 5.4]{muger2003structure}, $\delta\otimes X\ncong X$ for every $X\in \Irr(\C)$. Therefore, $G[X]$ is trivial for every $X\in \Irr(\C)$.
\end{proof}

\subsection{Equivariantizations and de-equivariantizations}\label{sec23}
Let $\C$ be a fusion category with an action of a finite group $G$. We then can define a new fusion category $\C^G$ of $G$-equivariant objects in $\C$. An object of this category is a pair $(X,(u_g)_{g\in G})$, where $X$ is an object of $\C$, $u_g : g(X)\to X$ is an isomorphism for all $g\in G$, such that
$$u_{gh}\circ \alpha_{g,h} =u_g \circ g(u_h),$$
where $\alpha_{g,h}: g(h(X))\to gh(X)$ is the natural isomorphism associated to the action. Morphisms
and tensor product of equivariant objects are defined in an obvious way. This new category is called the $G$-equivariantization of $\C$.

In the other direction, let $\C$ be a fusion category and let
$\E = \Rep(G)\subseteq \mathcal{Z}(\C)$ be a Tannakian subcategory that embeds into $\C$ via the forgetful
functor $\mathcal{Z}(\C)\to \C$. Here $\mathcal{Z}(\C)$ denotes the Drinfeld center of $\C$. Let $A={\rm Fun}(G)$ be the algebra of functions on $G$. It is a commutative algebra in $\mathcal{Z}(\C)$. Let $\C_G$ denote the category of left $A$-modules in $\C$. It is a fusion category and called the de-equivariantization of $\C$ by $\E$. See \cite{drinfeld2010braided} for details on equivariantizations and de-equivariantizations.

Equivariantizations and de-equivariantizations are inverse to each other:
$$(\C_G)^G\cong\C\cong(\C^G)_G,$$
and their Frobenius-Perron dimensions have the following relations:
\begin{equation}\label{eq11}
\begin{split}
\FPdim(\C)=|G|\FPdim(\C_G) \,\mbox{\,and}\, \FPdim(\C^G)=|G| \FPdim(\C).
\end{split}
\end{equation}

\section{Structure of generalized near-group fusion categories}\label{sec3}
In the rest of this paper, we assume that the fusion categories involved is not pointed, unless other stated.

\medbreak
Let $\C$ be a fusion category. Recall from Section \ref{sec2.1} that $G:=G(\C)$ acts on $\Irr(\C)$ by left tensor product.

\begin{definition}
A generalized near-group fusion category is a fusion category $\C$ such that $G$ transitively acts on the set $\Irr(\C)-G$.
\end{definition}

Let $\C$ be a generalized near-group fusion category. For simplicity, we assume that $\{X_1,\cdots,X_n\}=\Irr(\C)-G$ is a full list of non-isomorphic non-invertible simple objects of $\C$. By equation \ref{decom1}, we may assume
\begin{equation}\label{decom2}
\begin{split}
X_1\otimes X_1^*=\bigoplus_{h\in \Gamma}h\oplus k_1X_{1}\oplus \cdots\oplus k_nX_{n},
\end{split}
\end{equation}
where $\Gamma=G[X_1]$ is the stabilizer of $X_1$ under the action of $G$, $k_1,\cdots, k_n$ are non-negative integers.

\medbreak

Lemma \ref{fusionrules3} and Proposition \ref{subcategory} are also obtained in \cite{Thornton2012Generalized}. We include their proofs for the sake of completeness
\begin{lemma}\label{fusionrules3}
Let $\C$ be a generalized near-group fusion category. Then the fusion rules of $\C$ are determined by:

(1)\, For any $1\leq i\leq n$, we have
\begin{equation}
\begin{split}
X_i\otimes X_i^*=X_1\otimes X_1^*.
\end{split}\nonumber
\end{equation}

(2)\, For any $1\leq i,j\leq n$, there exists $g_{ij}\in G$ such that
\begin{equation}
\begin{split}
X_i\otimes X_j=\bigoplus_{h\in \Gamma}g_{ij}h\oplus k_1(g_{ij}\otimes X_{1})\oplus \cdots\oplus k_n (g_{ij}\otimes X_{n}).
\end{split}\nonumber
\end{equation}
\end{lemma}
\begin{proof}
(1)\, Since $G$ transitively acts on $\Irr(\C)-G(\C)$, there exists $g_i\in G$ such that $X_i^*=g_i\otimes X_1^*$ for any $i$. Then
\begin{equation}
\begin{split}
X_i\otimes X_i^*&\cong X_i^{**}\otimes X_i^*\cong(g_i\otimes X_1^*)^*\otimes (g_i\otimes X_1^*)\\
&\cong X_1\otimes g_i^*\otimes g_i\otimes X_1^*\cong X_1\otimes X_1^*.
\end{split}\nonumber
\end{equation}
(2)\, For any $i,j$, there exist $g_{ij}\in G$ such that $X_i\cong g_{ij}\otimes X_j^*$. Then
\begin{equation}
\begin{split}
X_i\otimes X_j&\cong g_{ij}\otimes X_j^*\otimes X_j\cong g_{ij}\otimes(\bigoplus_{h\in \Gamma}h\oplus k_1X_{1}\oplus \cdots\oplus k_nX_{n})\\
&\cong\bigoplus_{h\in \Gamma}g_{ij}h\oplus k_1(g_{ij}\otimes X_{1})\oplus \cdots\oplus k_s (g_{ij}\otimes X_{n}).
\end{split}\nonumber
\end{equation}
\end{proof}

Let $G,\Gamma$ and $k_1,\cdots,k_n$ be the data associated to $\C$ as in Lemma \ref{fusionrules3}. We shall say $\C$ is a generalized near-group fusion category of type $(G,\Gamma,k_1,\cdots,k_n)$.

\begin{proposition}\label{normalsubgroup}
Let $\C$ be a generalized near-group fusion category of type $(G,\Gamma,k_1,\cdots,k_n)$. Then

(1)\, $\Gamma$ is a normal subgroup of $G$.

(2)\, $\Irr(\C)=G\cup \{X_s|s\in G/\Gamma\}$, where $X_{\overline{g}}=g\otimes X_1$, $g\in G$.

(3)\, The rank of $\C$ is $[G:\Gamma](1+|\Gamma|)$ and $\FPdim(\C)=[G:\Gamma](\FPdim(X)^2+|\Gamma|)$.
\end{proposition}
\begin{proof}
(1)\, By Lemma \ref{fusionrules3}, $G[g\otimes X_1]=G[X_1]=\Gamma$ for any $g\in G$. On the other hand, Lemma \ref{conjugate} shows that $G[g\otimes X_1]=gG[X_1]g^{-1}=g\Gamma g^{-1}$. Hence $\Gamma$ is normal in $G$.

(2)\, Let $X_{\overline{g}}=g\otimes X_1$ for every $\overline{g}\in G/\Gamma$. Since $\Gamma=G[X_1]$, we have $g\otimes X_1\cong h\otimes X_1$ if and only if $h^{-1}g\otimes X_1\cong X_1$ if and only if $h^{-1}g\in \Gamma$ if and only if $\overline{g}=\overline{h}$ in $G/\Gamma$. Hence the isomorphic class of $X_{\overline{g}}$ is well defined.

(3)\, Part (3) follows from Part (2).
\end{proof}

\begin{remark}\label{T_Y}
Let $\C$ be a generalized near-group fusion category of type $(G,\Gamma,k_1,\cdots,k_n)$.

(1)\, If $(k_1,\cdots,k_n)=(0,\cdots,0)$ then $X_i\otimes X_j$ is a direct sum of invertible simple objects by Lemma \ref{fusionrules3}. If this is the case then $\C$ is a generalized Tambara-Yamagami fusion category introduced in \cite{liptrap2010generalized}. In fact, it is easily observed that $\C$ is a generalized Tambara-Yamagami fusion category if and only if $(k_1,\cdots,k_n)=(0,\cdots,0)$.

(2)\, If $\C$ exactly has one non-invertible simple object, then $G=\Gamma$ and $\C$ is a near-group fusion category introduced in \cite{siehler2003near}.
\end{remark}

\begin{proposition}\label{subcategory}
Let $\C$ be a generalized near-group fusion category of type $(G,\Gamma,k_1,\cdots,k_n)$. Assume that $\D$ is a non-pointed fusion subcategory of $\C$. Then $\D$ is also a generalized near-group fusion category.
\end{proposition}
\begin{proof}
It suffices to prove that the group $G(\D)$ generated by invertible simple objects of $\D$ transitively acts on $\Irr(\D)-G(\D)$. Let $X_i$ and $X_j$ be non-invertible simple objects in $\D$. Then there exists $g\in G$ such that $X_j=g\otimes X_i$. From $\dim\Hom(X_j,g\otimes X_i)=\dim\Hom(g,X_j\otimes X_i^*)=1$, we know that $g$ is a summand of $X_j\otimes X_i^*$. On the other hand, $X_j\otimes X_i^*$ lies in $\D$ since $\D$ is a fusion subcategory of $\C$. Hence $g$ is an element of $G(\D)$. This proves that $G(\D)$ transitively acts on $\Irr(\D)-G(\D)$
\end{proof}

%Let $1=d_0, d_1,\cdots, d_s$ be positive real numbers with $1 = d_0 < d_1< \cdots < d_s$, and $n_0,n_1,\cdots,n_s$ be positive integers.
%
%\begin{definition}
%Let $\D$ be an abelian subcategory of a fusion category. $\D$ is called of  category type $(d_0,n_0;d_1,n_1;\cdots;d_s,n_s)$ if $n_i$ is the number of the non-isomorphic simple objects in $\D$ of Frobenius-Perron dimension $d_i$, for all $i = 0,\cdots,s$.
%\end{definition}

\begin{proposition}\label{categorytype}
Let $\C$ be a generalized near-group fusion category of type $(G,\Gamma,k_1,\cdots,k_n)$. Assume that $(k_1,\cdots,k_n)\neq (0,\cdots,0)$. Then

(1)\, The adjoint subcategory $\C_{ad}$ is non-pointed. There is a 1-1 correspondence between the non-pointed fusion subcategories of $\C$ and the subgroups of the universal grading group $\mathcal{U}(\C)$.

(2)\, Every component $\C_g$ of the universal grading at least contains one invertible simple object $\delta$. In particular, $\Irr(\C_g)=\{\delta\otimes X_{i_1},\cdots,\delta\otimes X_{i_s} \}$, where $\Irr(\C_{ad})=\{X_{i_1},\cdots,X_{i_s} \}$.
\end{proposition}
\begin{proof}
(1)\, For every non-invertible simple object $X\in \C$, Lemma \ref{fusionrules3} shows that
\begin{equation}
\begin{split}
X\otimes X^*=\bigoplus_{h\in \Gamma}h\oplus k_1X_{1}\oplus \cdots\oplus k_nX_{n}.
\end{split}\nonumber
\end{equation}
Hence the adjoint subcategory $\C_{ad}$ is generated by $\Gamma$ and $X_i$'s with $k_i\neq 0$. Let $\D$ be a non-pointed fusion subcategory of $\C$ and let $Y\in \D$ be a non-invertible simple object. Then $X\otimes X^*=Y\otimes Y^*$ by Lemma \ref{fusionrules3}. Since $\D$ is a fusion subcategory, we get that $Y\otimes Y^*$ and hence $\Gamma$ and $X_i$'s with $k_i\neq 0$ are contained in $\D$. So $\C_{ad}$ is a fusion subcategory of $\D$. This shows that every non-pointed fusion subcategory of $\C$ contains $\C_{ad}$. Therefore, part (1) follows from \cite[Corollary 2.5]{drinfeld2010braided}.

(2)\, By assumption, $\C_{ad}$ contains a non-invertible simple object $Y$. Let $X$ be a simple object in $\C_g$. We may assume that $X$ is not invertible. Then $X\otimes Y\in \C_g\otimes \C_{ad}\subseteq \C_g$. By Lemma \ref{fusionrules3}(2), $X\otimes Y$ contains $|\Gamma|$ invertible simple objects. Hence $\C_g$ at least contains an invertible simple object.

Let $\delta\in\C_g$ be an invertible simple object, and $X_{i_1},\cdots,X_{i_s}$ be all non-isomorphic simple objects in $\C_{ad}$. Then $\delta\otimes X_{i_1},\cdots,\delta\otimes X_{i_s}$ are non-isomorphic simple objects in $\C_g$. Since
$$\FPdim(\delta\otimes X_{i_j})=\FPdim(\delta)\FPdim(X_{i_j})=\FPdim(X_{i_j}),$$
a dimension counting of $\C_g$ and $\C_{ad}$ shows that $\delta\otimes X_{i_1},\cdots,\delta\otimes X_{i_s}$ are all non-isomorphic simple objects in $\C_g$. This completes the proof.
\end{proof}

\begin{remark}\label{remark1}
Galindo \cite{GALINDO2011233} introduced the notion of a crossed product tensor category  and gave a description of this class of tensor categories, graded monoidal functors, monoidal natural transformations, and braidings in terms of coherent outer $G$-actions over tensor categories. By definition, a graded tensor category over a group $G$ is called a crossed product tensor category if every homogeneous component contains at least one invertible object. Proposition \ref{categorytype} shows that a generalized near-group fusion category of type $(G,\Gamma,k_1,\cdots,k_n)$ is a crossed product fusion category if $(k_1,\cdots,k_n)\neq (0,\cdots,0)$.
\end{remark}

\begin{theorem}\label{structure}
Let $\C$ be a generalized near-group fusion category of type $(G,\Gamma,k_1,\cdots,k_n)$. Then one of the following holds.

(1)\, If $(k_1,\cdots,k_n)=(0,\cdots,0)$ then $\C$ is a generalized Tambara-Yamagami fusion category.

(2)\, If $(k_1,\cdots,k_n)\neq (0,\cdots,0)$ then $\C$ is an extension of a smaller generalized near-group fusion category with trivial universal grading.
\end{theorem}

\begin{proof}
The first part is clear, it suffices to prove the second part.

If $(k_1,\cdots,k_n)\neq (0,\cdots,0)$ then the adjoint subcategory $\C_{ad}$ is not pointed and hence it is the smallest non-pointed fusion subcategory of $\C$ by Proposition \ref{categorytype}(1). This is because that $\C_{ad}$ corresponds to the trivial subgroup of $\mathcal{U}(\C)$. Let $X\in \C_{ad}$ be a non-invertible simple object. Then Lemma \ref{fusionrules3} shows the decomposition of $X\otimes X^*$ contains non-invertible simple objects. Hence $(\C_{ad})_{ad}$ is not pointed. But we have shown that $\C_{ad}$ is the smallest non-pointed fusion subcategory of $\C$. Hence $\C_{ad}=(\C_{ad})_{ad}$ and the universal grading group $\mathcal{U}(\C_{ad})$ of $\C_{ad}$ is trivial.
\end{proof}

Theorem \ref{structure} shows that a full classification of generalized near-group fusion categories is reduced to classifying such fusion categories with trivial universal grading. We recall two examples of generalized near-group fusion categories with trivial universal grading.

\begin{example}
 (Near-group fusion categories \cite{siehler2003near}) Let $G$ be a finite group and $\kappa$ be a nonnegative integer. A near-group fusion category of type $(G,\kappa)$ is a fusion category $\C$ whose isomorphism classes of simple objects are given by $G$ and a non-invertible object $X $, satisfying
 $$g\otimes h=gh,\qquad X\otimes X=\bigoplus_{g\in G}g\oplus \kappa X,\quad \forall g,h\in G.$$

A near-group fusion category $\C$ of type $(G,\kappa)$ with $\kappa>0$ admits no faithful grading. A near-group fusion category $\C$ of type $(G,0)$ is a Tambara-Yamagami category. In this case, $\C$ admits a faithful $\mathbb{Z}_2$-grading.
\end{example}

\begin{example}
 (Super-modular categories $PSU(2)_6$ \cite{Bruillard2017Fermionic}) Super-modular categories $PSU(2)_6$ are the adjoint subcategories of the modular categories $SU(2)_6$. They are non-split; that is, they can not be decomposed into a Deligne's tensor product of a modular category and $\svect$.

We denote the simple objects of a super-modular categories $PSU(2)_6$ by $\1,\delta, X, Y$. Its fusion rules are listed below.
\begin{center}
\begin{tabular}{ccc}
$\delta\otimes \delta=\1$,\quad $\delta\otimes X=Y$,\quad $\delta\otimes Y=X$,\\
$X\otimes X=Y\otimes Y=\1\oplus X\oplus Y$, $X\otimes Y=Y\otimes X=\delta\oplus X\oplus Y$.
\end{tabular}
\end{center}
It is clear that it is a generalized near-group fusion categories with trivial universal grading.
\end{example}

\begin{remark}\label{exten-GNG}
We notice that the opposite direction of Theorem \ref{structure}(2) does not always hold. That is, an extension of a generalized near-group fusion category is not necessary a generalized near-group fusion category. By \cite[Section III.G]{Bruillard2017Fermionic}, there exists a rank $7$ modular category $\C=\C_0\oplus C_1$ whose trivial component $\C_0$ is a super-modular categories $PSU(2)_6$. Since $PSU(2)_6$ has rank $4$, $\C_1$ has rank $3$. Hence $\C$ is not a generalized near-group fusion category by Proposition \ref{categorytype}(2).
\end{remark}

\begin{lemma}\label{centralizer2}
Let $\C$ be a braided fusion category such that the universal grading group $\mathcal{U}(\C)$ is trivial.  Then we have

(1)\, $(\C_{pt})'=\C$.

(2)\, $\C_{pt}\subseteq \mathcal{Z}_2(\C)$.

(3)\, $\C$ contains $\svect$ if and only if $\mathcal{Z}_2(\C)$ contains $\svect$.
\end{lemma}
\begin{proof}
(1)\, Since the universal grading group is trivial, we have $\C_{ad}=\C$. By \cite[Proposition 2.1]{DNS2019}, $\C_{ad}\subseteq(\C_{pt})'$. Hence $\C\subseteq (\C_{pt})'$. On the other hand, $(\C_{pt})'$ is a fusion subcategory of $\C$. So we have $(\C_{pt})'=\C$.

(2)\, By \cite[Corollary 3.11]{drinfeld2010braided}, we have $(\C_{pt})''=\mathcal{Z}_2(\C)=\C_{pt}\vee \mathcal{Z}_2(\C)$, which shows that $\C_{pt}\subseteq \mathcal{Z}_2(\C)$.

(3)\, Suppose first that $\svect\subseteq \C$. Since $\C_{pt}$ is the largest pointed fusion subcategory of $\C$ and $\svect$ is pointed, we have $\svect\subseteq \C_{pt}$.  By part (2), $\svect$ is contained in $\mathcal{Z}_2(\C)$. The other direction is obvious.
\end{proof}

In the proposition below, we give a description of M\"{u}ger center of a braided generalized near-group fusion category with trivial universal grading.
\begin{proposition}\label{Trivial1}
Let $\C$ be a braided generalized near-group fusion category of type $(G,\Gamma,k_1,\cdots,k_n)$. Assume that $\C$ is not symmetric and the universal grading group $\U(\C)$ is trivial. Then $\mathcal{Z}_2(\C)=\C_{pt}$.
\end{proposition}
\begin{proof}
Let $X$ be a non-invertible simple object of $\C$, and let $\C\langle X\rangle$ be the fusion subcategory generated by $X$. By Remark \ref{remark1}(1), $\C$ does not contain proper non-pointed fusion categories. Hence $\C=\C\langle X\rangle$ since $\C\langle X\rangle$ is not pointed. In other words, every non-invertible simple object can generate $\C$.

Since we have assumed that $\C$ is not symmetric, the argument above implies that the M\"{u}ger center $\mathcal{Z}_2(\C)$ can not contain non-invertible simple objects, hence it is a fusion subcategory of $\C_{pt}$. On the other hand, Lemma \ref{centralizer2} shows that $\C_{pt}$ is a fusion subcategory of $\mathcal{Z}_2(\C)$ since $\U(\C)$ is trivial. Hence $\mathcal{Z}_2(\C)=\C_{pt}$.
\end{proof}

\section{Extensions of a Fibonacci category}\label{sec4}
In this section we will present one example of generalized near-group fusion categories which is an extension of a rank $2$ fusion category.

\medbreak
Ostrik classified rank $2$ fusion categories in \cite{ostrik2003fusion}. Let $\C$ be a rank $2$ fusion category with $\Irr(\C)=\{\1,X\}$. The possible fusion rules for $\C$ are:

$${\rm (1)}\ X\otimes X\cong \1; \quad  {\rm (2)}\ X\otimes X\cong \1\oplus X.$$

If the first possibility holds true then $\C$ is pointed, and hence equivalent to $\vect_{\mathbb{Z}_2}^{\omega}$ for some $\omega\in H^{3}({\mathbb{Z}_2},k^{*})={\mathbb{Z}_2}$. Hence there are two such categories.

If the second possibility holds true then the fusion rules of $\C$ are called the Yang-Lee rules. If this is the case, we call $\C$ a Fibonacci category. There are two such categories. They are both non-degenerate braided fusion categories.

\begin{lemma}\label{equations}
Let $a,b,c,d,e\geq 3$ be unknowns. Then

(1)\, Equation $$\cos^2\frac{\pi}{a}+\cos^2\frac{\pi}{b}=\frac{5+\sqrt{5}}{8}$$ has unique integral solutions $a=3,b=5$.

(2)\, Equation $$\cos^2\frac{\pi}{c}+\cos^2\frac{\pi}{d}+\cos^2\frac{\pi}{e}=\frac{5+\sqrt{5}}{8}$$ has no integral solutions.
\end{lemma}
\begin{proof}
The function $f(x)=\cos^2\frac{\pi}{x}$ $(x\geq 3)$ is an increasing function on $x$, and $f(10)=\frac{5+\sqrt{5}}{8}$. An easy examination from $x=3$ to $x=10$ proves the lemma.
\end{proof}

\medbreak
A fusion category is said of category type $(d_0,n_0; d_1,n_1;\cdots;d_s,n_s)$ if $n_i$ is the number of the non-isomorphic simple objects of Frobenius-Perron dimension $d_i$, for all $0\leq i\leq s$, where $1 = d_0 < d_1< \cdots < d_s$ are positive real numbers, and $n_0,n_1,\cdots,n_s$ are positive integers. 

\begin{theorem}\label{Exten_YL_Cat}
Let $\C=\oplus_{g\in G}\C_g$ be an extension of a Fibonacci category $\F$. Then

(1)\, $\C$ is of type $(1,n;\frac{1+\sqrt5}{2},n)$, where $n=|G(\C)|$.

(2)\, For every $g\in G$, $\rk(\C_g)=2$. Write $\Irr(\C_g)=\{\delta_g,Y_g\}$. Then $\FPdim(\delta_g)=1$ and $\FPdim(Y_g)=\frac{1+\sqrt5}{2}$.

(3)\, $\F=\C_{ad}$, $G=\U(\C)$ and the order of $\mathcal{U}(\C)$ is $n$.
\end{theorem}
\begin{proof}
Since  $\F=\C_e$ and the grading is faithful, we have $\FPdim(\C_g)=\FPdim(\F)=\frac{5+\sqrt5}{2}$, for all $g\in G$. This implies that $\rk(\C_g)\leq 3$ for all $g\in G$. Set $\Irr(\F)=\{1,Y\}$ with $\FPdim(Y)=\frac{1+\sqrt5}{2}$ .

\medbreak
If there exists $e\neq g\in G$ such that $\rk(\C_g)=1$ then we set $\Irr(\C_g)=\{Y_g\}$. Then $Y\otimes Y_g \in \C_e\otimes \C_g\subseteq \C_g$. Since $Y_g$ is the unique (non-isomorphic) simple object in $\C_g$, we get that $Y\otimes Y_g\cong \FPdim(Y)Y_g$, which implies that $\FPdim(Y)$ is integral, a contradiction.

\medbreak
If there exists $e\neq g\in G$ such that $\rk(\C_g)=2$ then we set $\Irr(\C_g)=\{\delta_g,Y_g\}$. We may reorder $\delta_g,Y_g$ such that $\FPdim(\delta_g)\leq\FPdim(Y_g)$. Since $\FPdim(\F)= \FPdim(\C_g)$, we have $\FPdim(\delta_g)^2+\FPdim(Y_g)^2=\frac{5+\sqrt5}{2}$, which implies that $\FPdim(\delta_g)<2$ and $\FPdim(Y_g)<2$. By \cite[Remark 8.4]{etingof2005fusion}, there exist integers $a,b\geq 3$ such that $\FPdim(\delta_g)=2\cos\frac{\pi}{a}$ and $\FPdim(Y_g)=2\cos\frac{\pi}{b}$. Hence we have equation  $4\cos^2\frac{\pi}{a}+4\cos^2\frac{\pi}{b}=\frac{5+\sqrt{5}}{2}$. Lemma \ref{equations} shows that $a=3$ and $b=5$. So $\FPdim(\delta_g)=1$ and $\FPdim(Y_g)=\frac{1+\sqrt{5}}{2}$.

\medbreak
If there exists $e\neq g\in G$ such that $\rk(\C_g)=3$ then we set $\Irr(\C_g)=\{X_g,Y_g,Z_g\}$. Similarly, we have an equation ($c,d,e\geq 3$ are integers):

$$4\cos^2\frac{\pi}{c}+4\cos^2\frac{\pi}{d}+4\cos^2\frac{\pi}{e}=\frac{5+\sqrt{5}}{2}.$$

Lemma \ref{equations} shows that this is impossible. Therefore, every component $\C_g$ is of rank $2$ and $\C$ is of type $(1,n;\frac{1+\sqrt5}{2},n)$, where $n=|G(\C)|$. This proves part (1) and part (2).

Since $\C_{ad}$ is a fusion subcategory of $\F$ (see Subsection \ref{sec2.1}) and $\F$ does not have proper fusion subcategory, we have $\C_{ad}=\vect$ or $\F$. It is clear that $\C_{ad}$ can not be the trivial fusion category $\vect$, otherwise $\C$ is pointed. Hence $\C_{ad}=\F$ and $G=\mathcal{U}(\C)$ has order $n$.
\end{proof}

In the rest of this section, we will keep notation as in the proof of Theorem \ref{Exten_YL_Cat}.

\begin{corollary}\label{fusionrules}
Let $\C=\oplus_{g\in G}\C_g$ be an extension of a Fibonacci category $\F$. Then the fusion rules of $\C$ are:
$$Y_g\otimes Y_h\cong \delta_{gh}\oplus Y_{gh},$$
$$\delta_g\otimes Y_h\cong Y_{gh},Y_h\otimes \delta_g\cong Y_{hg},\delta_g\otimes \delta_h\cong \delta_{gh}.$$
\end{corollary}
\begin{proof}
Since $Y_g\otimes Y_h$ is contained in $\C_{gh}$ and $\C_{gh}$ only contains two non-isomorphic simple objects, a dimension counting shows that $Y_g\otimes Y_h\cong \delta_{gh}\oplus Y_{gh}$. To prove the remained isomorphisms, it suffices to notice that they are simple objects and contained in $\C_{gh}$, $\C_{hg}$ and $\C_{gh}$, respectively.
\end{proof}

\begin{remark}\label{exten-Fib}
(1)\, The corollary above shows that the action  of the group $G(\C)$ by left(or right) tensor multiplication on the set $\Irr(\C)-G(\C)$ is transitive. More precisely, $Y_h\cong \delta_{hg^{-1}}\otimes Y_g$ for all $g,h\in G$. Therefore, $\C$ is a generalized near-group fusion category.

(2)\, It follows from Corollary \ref{fusionrules} that the Grothendieck ring $K(\C)$ of $\C$ is commutative if and only if $G$ is commutative.
\end{remark}

It is easy to check that the map $f:\U(\C)\to G(\C)$ given by $f(g)=\delta_g$ is an isomorphism of groups, by Corollary \ref{fusionrules}. Hence we get the following corollary.

\begin{corollary}\label{universalGraGroup}
The universal grading group $\U(\C)$ is isomorphic to the group $G(\C)$.
\end{corollary}

%\begin{example}
%Let $\C$ be a Deligne's tensor product of a Fibonacci category $\F$ and a pointed fusion category $\B$. It is easy to check that $\C$ admits a faithful $G(\B)$-grading and $\C_e$ is tensor equivalent to $\F$.
%
%When $n=2$, the fusion category $\C$ is self-dual and admits a $\mathbb{Z}_2$-grading. The Grothendieck ring $K_0(\C)$ is the unique rank $4$ graded self-dual fusion ring which is categorifiable, see \cite{Dong2017Non}.
%\end{example}

%\begin{corollary}\label{non-pointed_fusionsubcategory}
%There is a 1-1 correspondence between the non-pointed fusion subcategories of $\C$ and the subgroups of the universal grading group $\mathcal{U}(\C)$. In particular, if $\mathcal{U}(\C)=\mathbb{Z}_p$ then $\C_{ad}=\F$ is the unique non-pointed fusion subcategories of $\C$, where $p$ is a prime number.
%\end{corollary}
%
%\begin{proof}
%Let $\D$ be a non-pointed fusion subcategory of $\C$. For every non-invertible simple object $X\in \D$, Lemma \ref{fusionrules} shows that
%\begin{equation}
%\begin{split}
%X\otimes X^*=\1\oplus Y.
%\end{split}\nonumber
%\end{equation}
%Hence $\C_{ad}$, generated by $\1$ and $Y$, is a fusion subcategory of $\D$. This shows that every non-pointed fusion subcategory of $\C$ contains $\C_{ad}$. On the other hand, any pointed fusion subcategory can not contain $\C_{ad}$. Therefore, the corollary follows from \cite[Corollary 2.5]{drinfeld2010braided}.
%\end{proof}
Let $\C$ be a fusion category, and let $\A, \B$ be fusion subcategories of $\C$. Let $\A\B$ be the full abelian (not necessarily tensor) subcategory of $\C$ spanned by direct summands in $X\otimes Y$, where $X\in \A$ and $Y\in \B$. We say that $\C$ factorizes into a product of $\A$ and $\B$ if $\C=\A\B$. A factorization $\C=\A\B$ of $\C$ is called exact if $A\cap \B=\vect$, and is denoted by $\C=\A\bullet\B$, see \cite{gelaki2017exact}.

By \cite[Theorem 3.8]{gelaki2017exact}, $\C=\A\bullet\B$ is an exact factorization if and only every simple object of $\C$ can be uniquely expressed in the form $X\otimes Y$, where $X\in \Irr(\A)$ and $\Irr(\B)$. If $\C$ is braided and $\C=\A\bullet\B$ admits an exact factorization then $\C=\A\boxtimes\B$ a Deligne tensor product, see \cite[Corollary 3.9]{gelaki2017exact}.

\begin{theorem}\label{extension_YL}
Let $\C=\oplus_{g\in G}\C_g$ be an extension of a Fibonacci category $\F$. Then $\C=\F\bullet \C_{pt}$ is an exact factorization of $\F$ and $\C_{pt}$.
\end{theorem}
\begin{proof}
By Theorem \ref{Exten_YL_Cat}, $\Irr(\C_{pt})=\{\delta_g|g\in G\}$. Then Corollary \ref{fusionrules} shows that every simple object in $\C$ admits the form $X\otimes Y$ where $X\in \F$ and $Y\in \C_{pt}$. The Theorem then follows from \cite[Theorem 3.8]{gelaki2017exact}.
\end{proof}

\begin{corollary} Let $\C$ be a braided fusion category. Suppose that
$\C=\oplus_{g\in G}\C_g$ is an extension of a Fibonacci category $\F$. Then $\C \cong \F\boxtimes \C_{pt}$ as braided fusion categories. In this case, $\mathcal{Z}_2(\C)=\mathcal{Z}_2(\C_{pt})$.
\end{corollary}
\begin{proof} It is enough to show the first statement.
Since Fibonacci categories are non-degenerate, then $\C \cong \F\boxtimes \F'$, by Theorem \ref{MugerThm}. Since $\F_{pt} \cong \vect$, then $\F' = \C_{pt}$, as was to be shown.
%The first statement follows from Theorem \ref{extension_YL} and \cite[Corollary 3.9]{gelaki2017exact}. The second statement follows from Proposition \ref{cad_non_deg}.
\end{proof}

\section{Slightly degenerate generalized near-group fusion categories}\label{sec5}
\begin{lemma}\label{dimuc}
Let $\C$ be a generalized near-group fusion category of type $(G,\Gamma,k_1,\cdots,k_n)$. Assume that $\FPdim(\C_{pt})=|\mathcal{U}(\C)|$ and $(k_1,\cdots,k_n)\neq(0,\cdots,0)$. Then $\C_{ad}$ is a Fibonacci category.
\end{lemma}
\begin{proof}
By Theorem \ref{categorytype}, every component $\C_g$ of the universal grading of $\C$ at least has one invertible simple object. Hence every component $\C_g$ exactly contains one invertible simple object by our assumption $\FPdim(\C_{pt})=|\mathcal{U}(\C)|$.

By Proposition \ref{normalsubgroup}, the number of non-isomorphic non-invertible simple objects is not more than the order of $G$. In addition, Theorem \ref{categorytype} shows that every component $\C_g$ admits the same type. Hence every component $\C_g$ only contains two simple objects: one is invertible and the other is not. In particular, $\C_{ad}$ is a Fibonacci category by the classification of rank $2$ fusion categories \cite{ostrik2003fusion}.
\end{proof}

\medbreak
The following corollary is also obtained in \cite{Thornton2012Generalized} under the assumption that the fusion categories involved are modular.

\begin{corollary}\label{non-deg}
Let $\C$ be a braided generalized near-group fusion category of type $(G,\Gamma,k_1,\cdots,k_n)$. Assume that $\C$ is non-degenerate. Then $\C$ fits into one of the following classes:

(1)\, $\C\cong \I\boxtimes \B$, where $\I$ is an Ising category, $\B$ is a pointed fusion category.

(2)\, $\C\cong \C_{ad}\boxtimes \C_{pt}$,  where $\C_{ad}$ is a Fibonacci category.
\end{corollary}

\begin{proof}
We first assume that $(k_1,\cdots,k_n)=(0,\cdots,0)$. Then Part (1) follows from Remark \ref{T_Y} and \cite[Theorem 5.4]{natale2013faithful}.

We then assume that $(k_1,\cdots,k_n)\neq(0,\cdots,0)$. Lemma \ref{non_degenerate} shows that $\FPdim(\C_{pt})=|\mathcal{U}(\C)|$. Hence $\C_{ad}$ is a Fibonacci category by Lemma \ref{dimuc}. Hence $\C\cong \C_{ad}\boxtimes \C_{ad}'$ by Theorem \ref{MugerThm}, where $\C_{ad}'=\C_{pt}$ by \cite[Corollary 3.27]{drinfeld2010braided}. Hence $\C\cong \C_{ad}\boxtimes \C_{pt}$.
\end{proof}

\medbreak
To prove the theorem below, we should recall the construction from  \cite{DNS2019}. For every $N\geq 1$, an $N$-Ising fusion category  is a graded extension of a pointed fusion category of rank 2 by the cyclic group of order $\mathbb Z_{2^N}$. In addition every $N$-Ising fusion category is strictly weakly integral. Its group of invertible objects is isomorphic to $\Zz_2 \times \Zz_{2^{N-1}}$ and it has $2^{N-1}$ simple objects of Frobenius-Perron dimension $\sqrt 2$, none of which is self-dual except in the case $N = 1$.

As graded extensions of $\vect_{\Zz_2}$, $N$-Ising fusion categories are parameterized by the integer $N$ and a $2^N$th root of unity $\zeta$. The corresponding category is denoted by $\Ii_{N, \zeta}$.

It was shown that a braided $N$-Ising fusion category is always prime, that is, it does not contain any non-trivial non-degenerate fusion subcategories. It was also shown  that with respect to any possible braiding, an $N$-Ising fusion category is non-degenerate if and only if $N = 1$. Moreover, a slightly degenerate braided $N$-Ising category exists if and only if $N > 2$.

As shown in \cite{EM}, when $N \geq 2$ there is another family of non-pointed $\Zz_{2^N}$-extensions  of $\vect_{\Zz_2}$. They are not equivalent to any $N$-Ising fusion category and  do not admit any braiding .

\medbreak
One of the main result in \cite{DNS2019}  is the following theorem:

\begin{theorem}\label{DNSmain}
Let $\C$ be a non-pointed braided fusion category and suppose that $\C$ is an extension of a rank 2 pointed fusion category. Then $\C$ is equivalent as a fusion category to $\Ii_{N, \zeta} \boxtimes \B$, for some $N\geq 1$, where $\zeta \in k^\times$ is a primitive $2^N$th root of 1,  and $\B$ is a pointed braided fusion category.
\end{theorem}

\begin{theorem}\label{slightly-deg1}
Let $\C$ be a braided generalized near-group fusion category of type $(G,\Gamma,k_1,\cdots,k_n)$. Assume that $(k_1,\cdots,k_n)=(0,\cdots,0)$ and $\C$ is slightly degenerate.  Then $cd(\C)=\{1,\sqrt{2}\}$ and $\C\cong \Ii_{N, \zeta} \boxtimes \B$, for some $N>2$, where $\zeta \in k^\times$ is a primitive $2^N$th root of 1,  and $\B$ is a pointed braided fusion category.
\end{theorem}

\begin{proof}
Since we assume that $(k_1,\cdots,k_n)=(0,\cdots,0)$, the adjoint subcategory $\C_{ad}$ is generated by $\Gamma$ and $\FPdim(X)=\sqrt{|\Gamma|}$ for all non-invertible simple object $X$ of $\C$. In particular, \cite[Proposition 5.2(ii)]{natale2013faithful} shows that
\begin{equation}\label{order}
\begin{split}
|\mathcal{U}(\C)|=2[G:\Gamma].
\end{split}
\end{equation}
By Proposition \ref{slig_degen}, $|G|=2|\mathcal{U}(\C)|$ or $|G|=|\mathcal{U}(\C)|$.

\medbreak
 Suppose that $|G|=2|\mathcal{U}(\C)|$.  In this case, equality \eqref{order} implies that $|\Gamma|=4$.  By Proposition \ref{slig_degen},  $\C_{ad}$ contains the M\"{u}ger center $\svect$ of $\C$. Let $\delta$ be the invertible simple object generating $\svect$. Then we may write $\Gamma=\{\1,\delta,g,h\}$. Hence $X\otimes X^*=\1\oplus \delta\oplus g\oplus h$ for any non-invertible simple object $X$. Then $\dim\Hom(\delta\otimes X,X)=\dim\Hom(\delta,X\otimes X^*)=1$, and therefore  $\delta\otimes X\cong X$. This contradicts \cite[Proposition 2.6(i)]{etingof2011weakly}.  This discards this possibility.

\medbreak
 Therefore $|G|=|\mathcal{U}(\C)|$. In this case, equality (\ref{order}) implies that $|\Gamma|=2$. Hence $cd(\C)=\{1,\sqrt{2}\}$ and $\C$ is an extension of a rank $2$ pointed fusion category. In particular, $\C$ is not pointed.  By Theorem \ref{DNSmain}, $\C\cong \Ii_{N, \zeta} \boxtimes \B$, for some $N$, where $\zeta \in k^\times$ is a primitive $2^N$th root of 1, and $\B$ is a pointed braided fusion category. The statement that $N>2$ follows from \cite[ Lemma 4.14]{DNS2019}.
\end{proof}

Before giving the proof of the following theorem, we recall the construction from \cite{bruillard2017classification}. The modular categories $SU(2)_{4k+2}$ are constructed as subquotient categories of representations of quantum groups $U_q\mathfrak{sl}_2$ with $q=e^{\frac{\pi}{4k+4}}$. Replacing $q$ by $q^t$ with $(t,4k + 4) = 1$, we can get new categories with the same fusion rules. These modular categories are denoted by $\C(\mathfrak{sl}_2, q^t, 4k + 4)$ and  their adjoint subcategories are denoted by $\C(\mathfrak{psl}_2, q^t, 4k + 4)$.  The following lemma is taken from \cite[Theorem 3.1]{bruillard2017classification}.

\begin{lemma}\label{bruill2017}
Any non-split super-modular category of rank $4$ is of the form $\C(\mathfrak{psl}_2, q^t, 8)$ with $(t,2)=1$.
\end{lemma}

Let $\Irr_{\alpha}(\C)$ be the set of nonisomorphic simple objects of Frobenius-Perron dimension $\alpha$.

\begin{lemma}\label{lemma100}
Let $\C$ be a braided fusion category. Suppose that the M\"{u}ger center $\Z_2(\C)$ contains the category $\svect$ of super vector spaces. Then the rank of $\Irr_{\alpha}(\C)$ is even for every $\alpha$.
\end{lemma}
\begin{proof}
Let $\delta$ be the invertible object generating $\svect$, and let $X$ be an element in $\Irr_{\alpha}(\C)$. Then $\delta\otimes X$ is also an element in $\Irr_{\alpha}(\C)$. By \cite[Lemma 5.4]{muger2000galois}, $\delta\otimes X$ is not isomorphic to $X$. This implies that $\Irr_{\alpha}(\C)$ admits a partition $\{X_1,\cdots,X_n\}\cup \{\delta\otimes X_1,\cdots,\delta\otimes X_n\}$. Hence the rank of $\Irr_{\alpha}(\C)$ is even.
\end{proof}

\begin{theorem}\label{slightly-deg2}
Let $\C$ be a braided generalized near-group fusion category of type $(G,\Gamma,k_1,\cdots,k_n)$. Assume that $(k_1,\cdots,k_n)\neq(0,\cdots,0)$ and $\C$ is slightly degenerate. Then $\C$ is exactly one of the following::

(1)\, $\C\cong \C_{ad}\boxtimes \C_{pt}$,  where $\C_{ad}$ is a Fibonacci category.

(2)\, $\C\cong\C_{ad}\boxtimes \B$,  where $\C_{ad}$ is a slightly degenerate fusion category of the form $\C(\mathfrak{psl}_2,q^t,8)$ with $q=e^{\frac{\pi i}{8}}$ and $(t,2)=1$, $\B$ is a non-degenerate pointed fusion category.
\end{theorem}
\begin{proof}
By Proposition \ref{slig_degen}, $\FPdim(\C_{pt})=|\mathcal{U}(\C)|$ or $\FPdim(\C_{pt})=2|\mathcal{U}(\C)|$.

Suppose that $\FPdim(\C_{pt})=|\mathcal{U}(\C)|$. In this case, $\C_{ad}$ is a Fibonacci category by Lemma \ref{dimuc}. Hence $\C\cong \C_{ad}\boxtimes \C_{ad}'$ by Theorem \ref{MugerThm}, where $\C_{ad}'=\C_{pt}$ by \cite[Corollary 3.29]{drinfeld2010braided}. Hence $\C\cong \C_{ad}\boxtimes \C_{pt}$. This proves Part (1).

\medbreak
Suppose that $\FPdim(\C_{pt})=2|\mathcal{U}(\C)|$. By Theorem \ref{categorytype}, every component $\C_g$ of the universal grading of $\C$ at least has one invertible simple object. Moreover, every component $\C_g$ admits the same type. Hence every component $\C_g$ exactly contains two invertible simple objects.

By Proposition \ref{normalsubgroup}, the number of non-isomorphic non-invertible simple objects is not more than the order of $G$. Hence the number of non-isomorphic non-invertible simple objects in $\C_g$ is $1$ or $2$.

\medbreak
If the number of non-isomorphic non-invertible simple objects in $\C_g$ is $1$  then $\C_{ad}$ is a fusion category of rank $3$. By Proposition \ref{slig_degen}, the M\"{u}ger center of $\C_{ad}$ contains the category $\svect$. This contradicts Lemma \ref{lemma100} which says that the rank of $\C_{ad}$ should be even.

\medbreak
If the number of non-isomorphic non-invertible simple objects in $\C_g$ is $2$ then $\C_{ad}$ is a rank $4$ fusion category. Let $\delta$ be the non-trivial invertible simple object in $\C_{ad}$, and $Y_1,Y_2$ be the non-invertible simple objects in $\C_{ad}$. Then $\delta$ generates the category $\svect$ by Proposition \ref{slig_degen}(2). By \cite[Lemma 5.4]{muger2000galois}, $\delta\otimes Y_i$ is not isomorphic to $Y_i$ for $i=1,2$. Hence $G[Y_i]$ is trivial and $\delta\otimes Y_i\cong Y_j$ for $i\neq j$.

The fact obtained above implies that if the M\"{u}ger center $\Z_2(\C_{ad})$ of $\C_{ad}$ contains $Y_1$ or $Y_2$ then $\Z_2(\C_{ad})=\C_{ad}$ is symmetric. Since $\C_{ad}$ contains $\svect$, $\C_{ad}$ is not Tannakian. In addition, $\FPdim(\C_{ad})>2$. Hence if $\C_{ad}$ is symmetric then it should admit a $\mathbb{Z}_2$-extension of a Tannakian subcategory by \cite[Corollary 2.50]{drinfeld2010braided}. This contradicts Remark \ref{remark1} which says the universal grading group of $\C_{ad}$ is trivial. This proves that $\Z_2(\C_{ad})$ can not contain $Y_1$ or $Y_2$. Hence $\Z_2(\C_{ad})=\svect$ and $\C_{ad}$ is slightly degenerate. By Lemma \ref{bruill2017}, $\C_{ad}$ is a fusion category of the form $\C(\mathfrak{psl}_2,q^t,8)$ with $q=e^{\frac{\pi i}{8}}$ and $(t,2)=1$.

By Proposition \ref{slig_degen}(2) and the arguments above, $\mathcal{Z}_2(\C_{ad})=\mathcal{Z}_2(\C_{ad}^{'})=\svect$. On the other hand, \cite[Proposition 3.29]{drinfeld2010braided} shows that $\C_{ad}^{'}=\C_{pt}$. Hence $\C_{pt}$ is slightly degenerate and admits a decomposition $\C_{pt}\cong\svect\boxtimes\B$ by \cite[Proposition 2.6(ii)]{etingof2011weakly}, where $\B$ is a non-degenerate pointed fusion category. So $\C$ admits a decomposition $\C\cong\B\boxtimes\B'$ by Theorem \ref{MugerThm}. Counting rank and Frobenius-Perron dimensions of simple objects on both sides, we obtain that $\B'$ is a rank $4$ non-pointed fusion category. By Remark \ref{remark1}, $\C_{ad}$ is the smallest non-pointed fusion subcategory of $\C$. Hence $\C_{ad}=\B'$. This proves Part (2).
\end{proof}

Combing Theorems \ref{slightly-deg1} and \ref{slightly-deg2}, we obtain the classification of slightly degenerate generalized near-group fusion categories.

\begin{theorem}\label{mainth}
Let $\C$ be a slightly degenerate generalized near-group fusion category. Then $\C$ is exactly one of the following:

(1)\, $\C\cong \Ii_{N, \zeta} \boxtimes \B$, for some $N>2$, where $\zeta \in k^\times$ is a primitive $2^N$th root of 1,  and $\B$ is a pointed braided fusion category.

(2)\, $\C\cong \C_{ad}\boxtimes \C_{pt}$,  where $\C_{ad}$ is a Fibonacci category.

(3)\, $\C\cong\C_{ad}\boxtimes \B$,  where $\C_{ad}$ is a slightly degenerate fusion category of the form $\C(\mathfrak{psl}_2,q^t,8)$ with $q=e^{\frac{\pi i}{8}}$ and $(t,8)=1$, $\B$ is a non-degenerate pointed fusion category.
\end{theorem}

\begin{corollary}
Let $\C$ be a slightly degenerate generalized near-group fusion category. Suppose that the universal grading group $\U(\C)$ is trivial. Then $\C\cong\C(\mathfrak{psl}_2,q^t,8)$.
\end{corollary}
\begin{proof}
Since $\U(\C)$ is trivial, $\C$ can not be equivalent to $\Ii_{N, \zeta} \boxtimes \B$.  Then  $\C\cong \F$ or $\C(\mathfrak{psl}_2,q^t,8)$ by Theorem \ref{mainth}, where $\F$ is a Fibonacci category. But Proposition \ref{Trivial1} shows that $\mathcal{Z}_2(\C)=\C_{pt}=\svect$. Hence $\C\cong\C(\mathfrak{psl}_2,q^t,8)$.
\end{proof}

\section{Braided generalized Tambara-Yamagami fusion categories}\label{sec6}
The following lemma is direct.
\begin{lemma}\label{multiplicity}
Let $X$ be an object in a fusion category $\C$. If $X\otimes X^*$ contains $2^{2i-1}$ copies of trivial simple object $\1$ then $X$ at least contains $2^i$ simple objects (duplicate objects are numbered by multiplicity).
\end{lemma}

Recall from \cite[Section 5]{natale2013faithful} that if $\C$ is generalized Tambara-Yamagami fusion category then $\FPdim(X)=\FPdim(Y)$ for all non-invertible simple objects $X$ and $Y$. In particular, $\FPdim(X)=\sqrt{|\Gamma|}$, where $\Gamma$ is the stalizer of $X$ under the action of $G(\C)$. Thus the adjoint subcategory $\C_{ad}$ coincides with the fusion subcategory generated by $\Gamma$, and we have  $cd(\C)=\{1,\sqrt{|\Gamma|}\}$, where $cd(\C)$ is the set of Frobenius-Perron dimensions of simple objects in $\C$. The dimension of a generalized Tambara-Yamagami fusion category is always even, more precisely it is equal to $2|G(\C)|$. In fact, we have the following characterization .

\begin{lemma}\label{grading-GTY}
Let $\C$ be a fusion category. Then $\C$ is a generalized Tambara-Yamagami fusion category if and only if $|cd(\C)|=2$ and $\C=\C_0\oplus \C_1$ has a $\mathbb{Z}_2$-grading, where $\C_0$ contains all invertible simple objects, $\C_1$ contains all non-invertible simple objects.
\end{lemma}

Let $\C$ be a braided generalized Tambara-Yamagami fusion category. Then $\C$ is nilpotent. By \cite[Theorem 1.1]{drinfeld2007g}, there exist prime numbers $2=p_1<p_2<\cdots<p_r$ such that $\C=\C_{p_1}\boxtimes\C_{p_2}\boxtimes\cdots\boxtimes\C_{p_r}$, where  $\C_{p_i}$ is a braided fusion category of dimension $p_i^{n_i}$ for some $n_i>0$. Since $p_2,\cdots,p_r$ are odd primes,  $\C_{p_2},\cdots,\C_{p_r}$ are pointed fusion categories. Hence, $\C_2$ is the unique generalized Tambara-Yamagami fusion subcategory of $\C$. It follows that the classification of braided generalized Tambara-Yamagami fusion categories $\C$ is reduced to the case when $\FPdim(\C)$ is a power of $2$. 

Let $\C$ be braided generalized Tambara-Yamagami fusion categories of dimension $2^n$. By \cite[Theorem 2.11]{etingof2011weakly}, we may assume $\FPdim(X)=2^i$ if $\C$ is integral or $\FPdim(X)= 2^{j}\sqrt{2}$ if $\C$ is not integral, for every non-invertible simple object $X$ in $\C$, where $i\geq 1$, $j\geq 0$.

\begin{theorem}\label{BGTY}
Let $\C$ be a braided generalized Tambara-Yamagami fusion category of dimension $2^n$. Then

(1)\, Suppose that $\C$ is integral and  $cd(\C)=\{1,2^i\}$. Then $\C$ is a $G$-equivariantization of a pointed fusion category, where $G$ is an Abelian group of order $2^{2i-1}$. In particular, $\C$ is group-theoretical.

(2)\, Suppose that $\C$ is not integral and  $cd(\C)=\{1,2^i\sqrt{2}\}$. Then $C$ is a $G$-equivariantization of some $\Ii_{N, \zeta} \boxtimes \B$, where $G$ is an Abelian group of order $2^{2i-1}$,  and $\B$ is a braided pointed fusion category.
\end{theorem}
\begin{proof}
Suppose that $\E=\Rep(G)\subset \C$ is  a Tannakian subcategory. Let  $\C_G$ be the de-equivariantization of $\C$ by $\E$ and $F:\C\to \C_G$ be the corresponding forgetful functor. Then $\E$ is the kernel of $F$; that is, $F(g)$ is some copies of the trivial object $\1$ for every simple object $g\in\E$.

Let $X$ be a non-invertible simple object in $\C$. Then 
$$X\otimes X^*=\sum_{g\in\E}g\oplus \sum_{h\in \Gamma-\E}h.$$
Applying the forgetful functor $F$, we have 

\begin{equation}\label{eqBGTY}
\begin{split}
F(X)\otimes F(X)^*=\underbrace{\1\oplus\cdots\oplus\1}_{\FPdim(\E)}\oplus \sum_{h\in \Gamma-\E}F(h).
\end{split}
\end{equation}

\medbreak
(1)\, The adjoint subcategory $\C_{ad}$ with dimension $2^{2i}$ is symmetric by Lemma \ref{centralizer_C_ad}, since it is pointed. By Theorem \ref{SymmCat}, we get a Tannakian subcategory $\E=\Rep(G)$ of dimension $\FPdim(\C_{ad})/2=2^{2i-1}$.

The Equation \ref{eqBGTY} shows that $F(X)$  is not simple and has at least $2^{i}$ simple objects by Lemma \ref{multiplicity}. The fact that $\FPdim(F(X))=\FPdim(X)= 2^{i}$ hence shows that $F(X)$ is a direct sum of $2^{i}$ invertible simple objects. So we get that $F(Z)$ is a direct sum of  invertible simple objects for any simple object $Z$ in $\C$. 

On the other hand, \cite[Lemma 4.6(iii)]{drinfeld2010braided} shows that every simple object in $\C_G$ is a direct summand of $F(Z)$ for some $Z\in \C$. We thus show that every simple object in $\C_{G}$ is invertible. That is, $\C_{G}$ is pointed. Hence $\C$ is a $G$-equivariantization of a pointed fusion category $\C_{G}$. Finally, $\C$ is group-theoretical by \cite[Theorem 7.2]{naidu2009fusion}.

\medbreak
(2)\, Assume $\C$ is of type $(1,m2^{2i+1};2^i\sqrt{2},m)$, see Lemma \ref{grading-GTY}. By \cite[Theorem 3.14]{drinfeld2010braided}, we have $\FPdim(\C_{ad})\FPdim((\C_{ad})')=\FPdim(\C)\FPdim(\C_{ad}\cap\C^{'})$.  Then $\FPdim((\C_{ad})')=2m\FPdim(\C_{ad}\cap\C^{'})$. By the proof of Lemma \ref{centralizer_C_ad}, $(\C_{ad})'$ contains $\C_{pt}$ as a fusion subcategory. Hence $\FPdim(\C_{ad}\cap\C^{'})\geq 2^{2i}$.

By Theorem \ref{SymmCat}, we get a Tannakian subcategory $\E=\Rep(G)$ of dimension $\FPdim(\C_{ad}\cap\C^{'})/2=2^{2i-1}$. Then Equation \ref{eqBGTY} shows that $F(X)$  is not simple and has at least $2^{i}$ simple objects. On the other hand, $\FPdim(F(X))=\FPdim(X)= 2^{i}\sqrt{2}$ which shows that $F(X)$ is a direct sum of $2^{i}$ simple objects with dimension $\sqrt{2}$. Again by \cite[Lemma 4.6(iii)]{drinfeld2010braided}, we know that the simple objects in $\C_G$ have dimension $1$ and $\sqrt{2}$. Hence $\C_G$ is an extension of a pointed fusion category of rank $2$, by \cite[Corollary 3.2]{DNS2019} or \cite[Corollary 3.3]{Dongpointed2020}. Finally, $\C_G$ is braided by \cite[Remark 2.3]{etingof2011weakly} since $\E$ is contained in $\C'$. It follows from Theorem \ref{DNSmain} that $\C_G\cong \Ii_{N, \zeta} \boxtimes \B$, for some $N\geq 1$, where $\zeta \in k^\times$ is a primitive $2^N$th root of 1,  and $\B$ is a pointed braided fusion category. 
\end{proof}

\begin{remark}
The result in Theorem \ref{BGTY} (1) can also be obtained by \cite[Theorem 6.10]{drinfeld2007g} and \cite[Theorem 7.2]{naidu2009fusion}. However, our result presents the explicit information on the group $G$.
\end{remark}

\section*{Acknowledgements}
We would like to thank Sonia Natale for very useful discussions and inspiration when she visited Nanjing University of Information Science and Technology. J. Dong is partially supported by the startup foundation for introducing talent of NUIST (Grant No. 2018R039) and the Natural Science Foundation of China (Grant No. 11201231).

% Used on Computer
%\bibliography{/users/dong/onedrive/documents/references/dongrefs}{}
%%\bibliographystyle{plain}
%\bibliographystyle{/Users/Dong/OneDrive/Documents/references/bibstyles/elsart-num-sort}

%bibliography{C:/Users/Dong/OneDrive/Documents/references/dongrefs}

% Used on Laptop
%\bibliographystyle{D:/oneDrive/Documents/references/bibstyles/elsart-num-sort}
%\bibliography{D:/oneDrive/Documents/references/dongrefs}

\end{document}